\documentclass{lms}

\usepackage{amsmath}
\usepackage{amssymb}
\usepackage{amsfonts}
\usepackage{color}
\usepackage{graphicx}
\usepackage{graphics}
\usepackage{epsfig}
\usepackage{verbatim}

\newtheorem{theorem}{Theorem}[section] 
\newtheorem{lemma}[theorem]{Lemma}     
\newtheorem{corollary}[theorem]{Corollary}
\newtheorem{proposition}[theorem]{Proposition}

\newnumbered{assertion}{Assertion}    
\newnumbered{conjecture}{Conjecture}  
\newnumbered{definition}{Definition}
\newnumbered{hypothesis}{Hypothesis}
\newnumbered{remark}{Remark}
\newnumbered{note}{Note}
\newnumbered{observation}{Observation}
\newnumbered{problem}{Problem}
\newnumbered{question}{Question}
\newnumbered{algoritheorem}{Algoritheorem}
\newnumbered{example}{Example}
\newunnumbered{notation}{Notation}

\newcommand{\CC}{{\mathbb C}}
\newcommand{\DD}{{\mathbb D}}
\newcommand{\TT}{{\mathbb T}}

\newcommand{\RR}{{\mathbb R}}
\newcommand{\NN}{{\mathbb N}}

\newcommand{\Pol}{{\mathcal P}}
\newcommand{\dist}{\mathrm{dist}}

\newcommand{\McC}{\raise.5ex\hbox{c}}

\title [Orthogonal polynomials, kernels, and zeros] {Orthogonal polynomials, reproducing kernels, and zeros of optimal approximants}

\author[B\'en\'eteau, Khavinson, Liaw, Seco, and Sola]{Catherine B\'en\'eteau, Dmitry Khavinson, Constanze Liaw, \\Daniel Seco, and Alan A. Sola}

\classno{42C05 (primary), 46E22, 47A16 (secondary)}

\begin{document}
\maketitle

\begin{abstract}
We study connections between orthogonal polynomials, reproducing kernel functions, and
polynomials $p$ minimizing Dirichlet-type norms $\|pf-1\|_{\alpha}$ for a given function $f$.
For $\alpha\in [0,1]$ (which includes the Hardy and Dirichlet spaces of the disk) and general $f$, we show that such extremal polynomials are non-vanishing in the closed unit disk.
For negative $\alpha$, the weighted Bergman space case, the extremal polynomials are non-vanishing on a disk of strictly smaller radius,
and zeros can move inside the unit disk. We also explain how $\mathrm{dist}_{D_{\alpha}}(1,f\cdot \mathcal{P}_n)$, where $\mathcal{P}_n$ is the space of polynomials of degree at most $n$, can be expressed in terms of quantities associated with orthogonal polynomials and kernels, and we discuss methods for computing the
quantities in question.
\end{abstract}


\section{Introduction}

The objective of this paper is to study the relationships between
certain families of orthogonal polynomials and other families of
polynomials associated with polynomial subspaces and shift-invariant subspaces in Hilbert spaces of functions on the unit disk $\mathbb{D}=\{z\in \CC\colon |z|<1\}$. We work in the setting of {\it
Dirichlet-type spaces} $D_{\alpha}$, $\alpha\in \mathbb{R}$, which consist
of all analytic functions $f=\sum_{k=0}^{\infty}a_kz^k$ on the unit disk
 satisfying
\begin{equation}
    \|f\|^2_\alpha=\sum_{k=0}^\infty (k+1)^{\alpha}|a_k|^2< \infty.
    \label{Dalphadef}
\end{equation}
Given also $g = \sum_{k=0}^{\infty}b_kz^k$ in $D_{\alpha}$, we have the associated inner product
\begin{equation}\label{ipseries}
     \langle f , g \rangle_{\alpha} = \sum_{k=0}^\infty (k+1)^{\alpha} a_k \overline{b_k} .
\end{equation}

We note that $D_{\beta}\subset
D_{\alpha}$ when $\beta\geq \alpha$.
The spaces $D_0$, $D_{-1}$, and $D_{1}$ coincide with the classical
Hardy space $H^2$, the Bergman space $A^2$, and the Dirichlet space $D$
of the disk $\mathbb{D}$ respectively. These important function spaces are discussed
in the textbooks \cite{Durbook} (Hardy), \cite{DS04,HKZ00} (Bergman), and \cite{EKMRBook} (Dirichlet). One can show that $D_{\alpha}$ are
algebras when $\alpha>1$, which makes the Dirichlet space an intriguing
borderline case. Each $D_{\alpha}$ is a
{\it reproducing kernel Hilbert space} (RKHS): for each $w\in \mathbb{D}$, there exists an
element $k_{\alpha}(\cdot,w)\in D_{\alpha}$, called the reproducing kernel,
such that
\begin{equation}
f(w)=\langle f, k_{\alpha}(\cdot,w)\rangle_{\alpha}
\label{repkerneldef}
\end{equation}
holds for any $f\in D_{\alpha}$. For instance, when $\alpha=-1$, this is the well-known Bergman kernel
$k(z,w)=(1-\bar{w}z)^{-2}$.

Given a function $f \in D_{\alpha}$, we are interested in finding
polynomial substitutes for $1/f$, in the following sense.
\begin{definition}\label{D-optapprox}
Let $f \in D_\alpha$. We say that a polynomial $p_n$ of degree at most $n\in \mathbb{N}$
is an {\it optimal approximant} of order $n$ to $1/f$ if $p_n$ minimizes
$\|p f-1\|_{\alpha}$ among all polynomials $p$ of degree at most $n$.
\end{definition}
It is clear that the polynomials $p_n$ depend on both $f$ and $\alpha$, but we suppress this dependence to lighten
notation. Note that for any function $f \in D_\alpha$, the optimal approximant $p_n$ ($n \in \NN$) exists and is unique, since $p_n f$ is the orthogonal projection of the function $1$ onto the finite dimensional subspace $f\cdot\Pol_{n}$, where $\Pol_{n}$ denotes the space of polynomials of degree at most $n$. Note that elements of $D_{\alpha}$ are {\it not} always invertible in the space: in
general, $1/f \notin D_{\alpha}$ when $f\in D_{\alpha}$. Thus, the
problem we are interested in is somewhat different from the
usual one of polynomial approximation in a Hilbert space of analytic functions.

Optimal approximants arise in the study of functions $f$ that are cyclic with respect to the {\it shift operator}
$S\colon f\mapsto zf$.
\begin{definition}
A function $f \in D_{\alpha}$ is said to be \emph{cyclic} in
$D_\alpha$ if the closed subspace generated by monomial multiples
of $f$,
\[ [f]=\overline{\textrm{span}\{z^kf\colon k=0,1,2,\ldots\}},\]
coincides with $D_{\alpha}$.
\end{definition}
No function that vanishes in the disk can be cyclic, since elements of $[f]$ inherit the zeros of $f$. The function $g \equiv 1$ is cyclic in all $D_{\alpha}$, and if
 a function $f$ is cyclic in $D_{\alpha}$, then it is cyclic in $D_{\beta}$ for all
$\beta\leq \alpha$. If $f$ is a cyclic function, then the optimal
approximants to $1/f$ have the property
\[\|p_n f-1\|_{\alpha} \longrightarrow 0, \quad n \rightarrow \infty,\]
and the $(p_n)$ yield the optimal rate of decay of these norms in
terms of the degree $n$. See \cite{BCLSS13,BS84} for more detailed
discussions of cyclicity. When $\alpha>1$, the algebra setting, cyclicity of $f$ is actually
equivalent to saying that $f$ is invertible, but there exist smooth
functions $f$ that are cyclic in $D_{\alpha}$ for $\alpha\leq 1$
{\it without having} $1/f\in D_{\alpha}$: functions of the form
$f=(1-z)^N$, $N \in \mathbb{N}$, furnish simple examples.

In the paper \cite{BCLSS13}, computations with optimal approximants resulted in the
determination of sharp rates of decay of the norms $\|p_nf - 1\|_{\alpha}$ for certain classes of functions
with no zeros in the disk but at least one zero on the unit circle $\TT$. Thus, the
polynomials $p_n$ are useful and we deem them worthy of further study. A number of interesting questions arise naturally.
For a given function $f$, what are the optimal approximants, and what is the rate of convergence of
$\|p_nf - 1\|_{\alpha}$?  How  are the zeros of the optimal approximants related to these rates, and does the location of the zeros of $p_n$ give any clues about whether a function $f$ is cyclic or not?

In \cite{BCLSS13} and \cite{FMS14}, it was explained that (see
\cite[Theorem 2.1]{FMS14} for the particular statement used here) the coefficients
$(c_k)_{k=0}^{n}$ of the $n$th optimal approximant are obtained by solving the linear system
\begin{equation}
M\vec{c}=\vec{e}_0,
\label{approxlinearsyst}
\end{equation}
with matrix $M$ given by
\[(M_{k,l})_{k,l=0}^n=(\langle z^kf, z^lf\rangle_{\alpha})_{k,l}\] and $\vec{e}_0=(\langle  1, f\rangle, \ldots, \langle 1,z^n f\rangle )^T=(\overline{f(0)}, \vec{0})$. For simple functions $f$, this system can be
solved in closed form for all $n$, leading to explicit expressions for $p_n$. In \cite{BCLSS13}, the authors
found the optimal approximants to $1/(1-z)$ for each $\alpha$, and plotted their respective zero sets;
a plot is reproduced in the next section. These plots, as well as the zero sets of optimal approximants for
other simple functions $f$ displayed in \cite{BCLSS13}, all had one thing in common: the zeros of the polynomials $p_n$, which we will denote by
  $\mathcal{Z}(p_n)$, were all outside the closed unit disk. Might this be true for any choice of $f$, or at least for $f$ non-vanishing in the disk---are optimal approximants always zero-free in the disk?

In this paper we give an answer to this question.
For non-negative $\alpha$, the answer is in the affirmative, in a strong sense: optimal approximants are always non-vanishing in the
closed disk, for essentially any $f\in D_{\alpha}$.
\begin{theorem*}[A]
Let $\alpha\geq 0$, let $f\in D_{\alpha}$ have $f(0)\neq 0$, and let $(p_n)$ be the optimal approximants to $1/f$. Then $\mathcal{Z}(p_n)\cap \overline{\mathbb{D}}=\varnothing$ for all $n$.
\end{theorem*}
For negative $\alpha$, there is still a closed disk on which no optimal approximant can vanish, but
this disk is strictly smaller than $\mathbb{D}$.
\begin{theorem*}[B]
Let $\alpha< 0$, let $f\in D_{\alpha}$ have $f(0)\neq 0$, and let $(p_n)$ be the optimal approximants to $1/f$. Then $\mathcal{Z}(p_n)\cap \overline{D}(0,2^{\alpha/2})=\varnothing$ for all $n$.
\end{theorem*}
We show that the radius $2^{\alpha/2}$ cannot be replaced by $1$, even if $f$ is assumed to be non-vanishing, by giving examples of cyclic functions $f\in D_{\alpha}$, $\alpha$ negative, such that $p_n(\lambda)=0$ for at least one $n$ and at least one $\lambda\in \mathbb{D}\setminus \overline{D}(0,2^{\alpha/2})$.

The proofs of these theorems rely on connections between the $p_n$, orthogonal polynomials in certain
weighted spaces determined by the given $f$, and reproducing kernel functions for the polynomial subspaces $f\cdot \mathcal{P}_n$. We also obtain conditions that relate
cyclicity of a given function $f$ to convergence properties of these orthogonal polynomials and
the reproducing kernel functions. For example, we show that a function $f$ is cyclic if and only if its
associated orthogonal polynomials $(\varphi_k)$ have $\sum_{k}|\varphi_k(0)|^2=|f(0)|^{-2}$.

The paper is structured as follows.
We begin Section \ref{s-EX} by revisiting the optimal
approximants to $1/(1-z)$ and by also examining the optimal approximants associated with
$f_N=(1-z)^N$, $N \in \NN$; the observations we make in this section
motivate much of the further development in the paper.

We point out a connection between the optimal approximants and
orthogonal polynomials in Section \ref{s-OP}. The starting point, given a function $f$ whose optimal approximants we wish to study, is to
introduce a modified space with inner product $\langle g,h\rangle_{\alpha,f}:=\langle fg, fh \rangle_{\alpha}$; for the Hardy and Bergman spaces this amounts to changing
Lebesgue measure (on $\TT$ or $\DD$ respectively) to weighted Lebesgue measure with weight
$|f|^2$. We study orthogonal polynomial bases for the subspace $f\cdot\Pol_{n}$ and obtain a formula for the optimal approximants in $D_\alpha$ in terms of the orthogonal polynomials (Proposition \ref{orthopolrepresentation}). For the Hardy space we show that this representation implies, via
known results concerning zero sets of orthogonal polynomials on the unit circle, that the optimal approximants do not have any roots in the closed disk (Theorem \ref{t-noroots}).

In Section \ref{s-RK} we examine reproducing kernels for the subspaces $f\cdot\Pol_{n}$. A relation between the reproducing kernel functions and the optimal approximants (see equation \eqref{optrep}) is key, and
allows us to prove our main result, Theorem \ref{norootsallalpha}.

By combining our results, we can characterize cyclicity of a function in the $D_\alpha$ spaces in terms of a pointwise (only) convergence property of the sum of absolute values of orthogonal polynomials; this is discussed in Section \ref{s-kernelscyclicity}.

Section \ref{s-Toeplitz} is devoted to a slightly different idea:
the formula \eqref{approxlinearsyst} requires the inversion of $n
\times n$ matrices $M$ with $(k,l)$-entry given by $M_{k,l} =
\left\langle z^kf, z^l f\right\rangle$. In the case $\alpha=0$, the
Hardy space, multiplication by $z^l$ is an isometry, and
$M_{k,l}= M_{k-l,0}$. Hence $M$ is a Toeplitz matrix. We use Levinson's algorithm for inverting Toeplitz
matrices to study optimal approximants, and we
revisit some of the results from the previous sections in the light of
this approach.

In the last section, Section \ref{s-Extraneous}, we discuss how zeros of optimal approximants can be computed in terms of inner products involving
the given function $f$, and produce examples of functions $f\in D_{\alpha}$, $\alpha$ negative, whose optimal approximants vanish inside the unit disk.

Some of the results we present here can be readily extended to more
general spaces of analytic functions, such as Bergman spaces with logarithmically subharmonic weights (see for example
\cite{HKZ00,W02,W03,DS04,CW13,FMS14}), but for simplicity, we will concentrate on the
$D_\alpha$-spaces as defined above. For convenience, we assume $f(0)\neq 0$ throughout
the paper; this simplifies the arguments and does not entail any substantial loss of generality.

\section{Motivating examples}\label{s-EX}
We begin by examining some functions with zeros on $\mathbb{T}$ that are cyclic, namely $f_N=(1-z)^N$, with $N \in \NN \backslash\{0\}$. We
present explicit formulas for optimal approximants to $1/f_N$ and investigate their properties, paving the way for further results and conjectures.

\begin{example}
For $f_1=1-z$, the optimal approximants to $1/f_1$ in $D_{\alpha}$
were found in \cite{BCLSS13}. Setting \[w_{\alpha}(k) = \|z^k\|_{\alpha}^2=(k+1)^{\alpha},\] the
optimal approximants $p_n$ are given by the corresponding Riesz
means of $n$th-order Taylor polynomials for $1/(1-z)=\sum_{k}z^k$.
In the series norm for $D_\alpha$ that we are considering
here, we have
\begin{equation}
p_n(z) = \sum_{k=0}^{n}\left(1-\frac{\sum_{j=k+1}^n
1/w_{\alpha}(j)}{\sum_{j=0}^{n+1} 1/w_{\alpha}(j)}\right)z^k. \label{cesaropoly}
\end{equation}

Using this formula, we can prove the following.
\begin{proposition}\label{p-EX}
Let $f(z) = 1-z$ and let $p_n$ denote the optimal approximants to $1/f$ in
$D_{\alpha}$.
\begin{itemize}
\item[(a)] The polynomials $p_n$ admit the following representation:
\begin{equation}\label{ruff1}
p_n(z)= \left(1-\frac{\sum_{k=0}^{n+1} z^k/w_{\alpha}(k)}{\sum_{k=0}^{n+1}
1/w_{\alpha}(k)} \right)/f(z)
\end{equation}
\item[(b)] The zero set of $p_n$ is given by $$\mathcal{Z}(p_n) =
\left\{z \neq 1: \sum_{k=0}^{n+1} z^k/w_{\alpha}(k) = \sum_{k=0}^{n+1} 1/w_{\alpha}(k)
\right\}.$$
\item[(c)] In the particular case of the Hardy space ($\alpha=0$) the polynomials admit an additional representation as follows:
\begin{equation}\label{ruff2}
p_n(z) = \frac{z^{n+2} - (n+2) z +n+1}{(n+2) (f(z))^2}.
\end{equation}
\end{itemize}
\end{proposition}

In particular, item (b) tells us that $\mathcal{Z}(p_n)$
does not intersect $\overline{\DD}$ for any $n$, confirming what Figure \ref{oneminuszeefig} suggests. Furthermore, an inspection of the formulas reveals that for
even $n$, the optimal approximants $p_n$ have no real roots, whereas
for odd $n$, the optimal approximant $p_n$ has exactly one real
root, which lies on the negative half-axis.

\begin{figure}
\includegraphics[width=0.45 \textwidth]{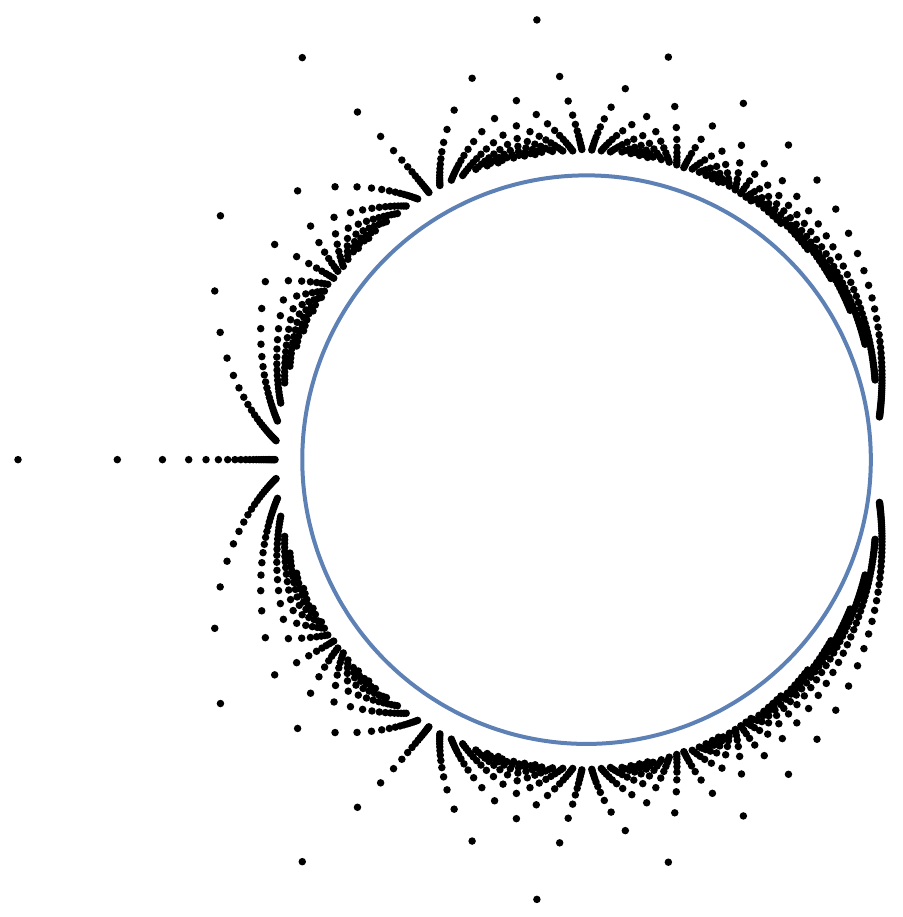}
\caption{Combined zero sets for optimal approximants to $1/(1-z)$ in $H^2$, for $n=0, \ldots, 50$.}
\label{oneminuszeefig}
\end{figure}

Our arguments below are elementary in nature, and clearly limited to this particular $f=1-z$, and similar functions. Nevertheless, the above observations provided some evidence in support of the notion that optimal approximants are
zero-free in the unit disk.

\begin{proof}
Parts (a) and (c) can be derived by long division of polynomials:
applying Ruffini's rule to the expression \eqref{cesaropoly} once
yields \eqref{ruff1} and using Ruffini's rule again on
\eqref{ruff1}, gives \eqref{ruff2}. Let us verify part (b). Using
\eqref{ruff1}, we see that $p_n(z)$ can only be zero at
singularities of $f$ or at points where $\sum_{k=0}^{n+1} z^k/w_k =
\sum_{k=0}^{n+1} 1/w_k$. Since $f$ is entire, and since from
\eqref{cesaropoly} we know that $p_n(1) \neq 0$, we have
$$\mathcal{Z}(p_n) \subset Z_n := \left\{z \neq 1:  \sum_{k=0}^{n+1} z^k/w_k = \sum_{k=0}^{n+1} 1/w_k
\right\}.$$ Whenever $z \in Z_n$ then the numerator in \eqref{ruff1}
is 0 and the denominator is not. Hence $\mathcal{Z}(p_n) = Z_n$.
\end{proof}
\end{example}

\begin{example}\label{f_a}
We now turn to $f_N=(1-z)^N$, $N \in \NN$ and $N \geq 2$, which has a multiple root at $\zeta=1$. The optimal approximants to
$1/f_N$ again admit an explicit representation in the case of the Hardy space.
If we let $B$ denote the beta function, $B(x,y)=\int_{0}^{1}t^{x-1}(1-t)^{y-1}dt$, then
the $\mathrm{n}$th-order optimal approximant to $1/(1-z)^N$ in $H^2$ is given by
\begin{equation}
p_n(z)=\sum_{k=0}^{n}\left( \binom{k+N-1}{k} \frac{B(n+N+1,N)}{B(n-k+1,N)}\right)z^k.
\label{squareoptimalapprox}
\end{equation}
 Once again, in Figure \ref{higherpowfig}, plots suggest that the zeros of the $H^2$-optimal approximants in lie outside the closed unit disk for
any power $N$. While this will turn out to be true, we shall see in Section \ref{s-Extraneous} that
the optimal approximants to $1/(1-z)^3$ in $D_{-2}$ do vanish in $\DD.$ 

A proof of Formula \ref{squareoptimalapprox} will be presented in the forthcoming paper \cite{Seco}, and it seems reasonable to suspect that the following holds.
\begin{conjecture}
The formula \eqref{squareoptimalapprox} remains valid for the optimal approximants to $1/f_a$ when
$f_a=(1-z)^a$, and $a \in \CC$ has positive real part.
Note that all such $f_a$ are cyclic in the Dirichlet space \cite[Proposition 13]{BS84}, and hence in all $D_{\alpha}$ for $\alpha \leq 1$.
\end{conjecture}

\begin{figure}
\includegraphics[width=0.4 \textwidth]{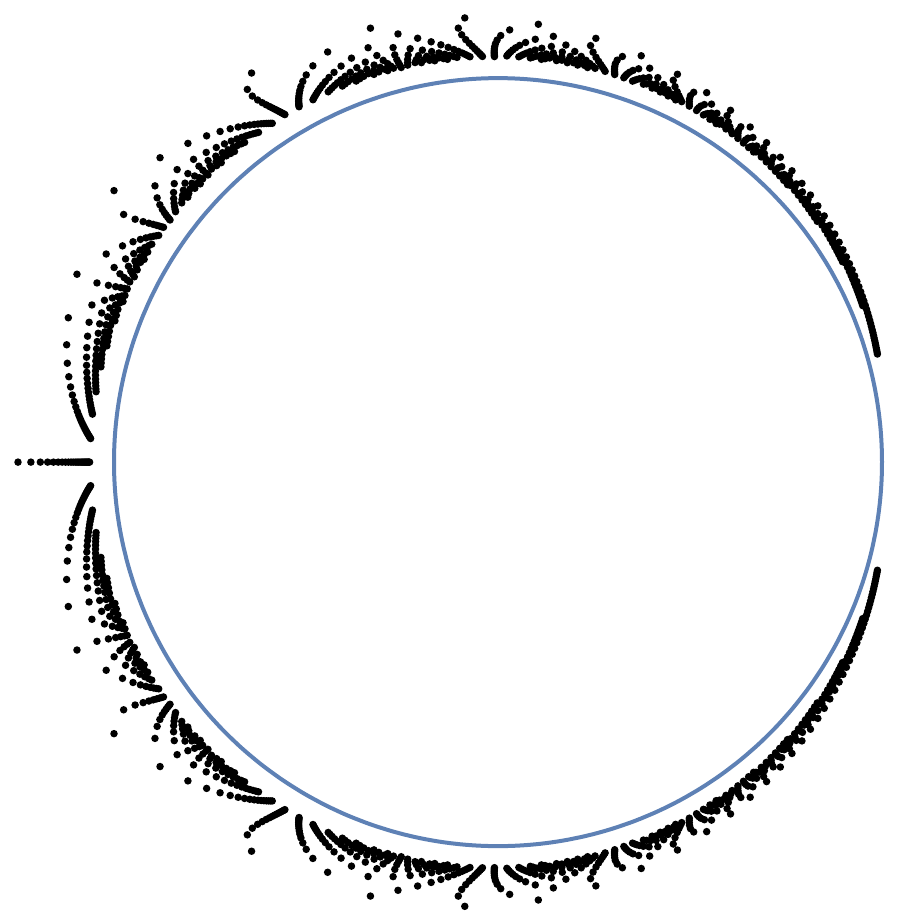}
\includegraphics[width=0.4 \textwidth]{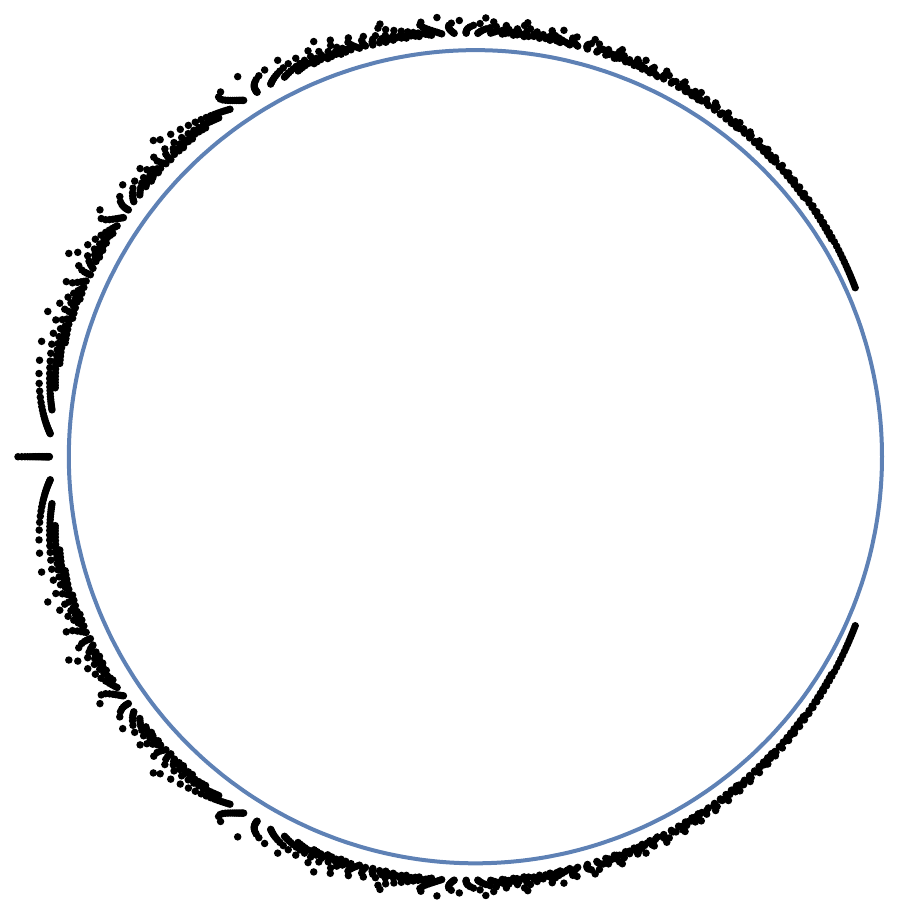}
\caption{Left: Combined zero sets for optimal approximants to $1/(1-z)^4$ for $n=1, \ldots, 50$. Right: Combined zero sets for optimal approximants to $1/(1-z)^8$ for
$n=1,\ldots, 50$.}
\label{higherpowfig}
\end{figure}
\end{example}

\section{Orthogonal Polynomials}\label{s-OP}

In order to generalize the observations of the preceding section to arbitrary functions, we now turn to a discussion of the relationship between optimal approximants and orthogonal polynomials.  Fix $\alpha \in \RR$, and let $f \in D_{\alpha}$, assuming $f(0)\neq 0$.  Consider the space $f\cdot\Pol_{n}$, where $\Pol_{n}$ is the space of polynomials of degree at most $n$.  If we let $\varphi_k  f$ be an orthonormal basis for the space $f\cdot\Pol_{n}$, where the degree of $\varphi_k$ is $k$, then for $ 0 \leq k \leq n$, the functions $\varphi_k$ satisfy
\[ \langle \varphi_k f , \varphi_j f \rangle_{\alpha} = \delta_{k,j}. \] In other words, we can think of the functions  $\varphi_k$ as being orthogonal polynomials in a ``weighted" $D_{\alpha,f}$ space
by defining an inner product of two functions $\varphi$ and $\psi$ in this weighted space by
\begin{equation}
 \langle \varphi , \psi \rangle_{\alpha,f} := \langle \varphi f , \psi f \rangle_{\alpha}.
\label{weightedspacedef}
\end{equation}
We let $\| \cdot\|_{\alpha,f}$ denote the corresponding weighted norm. Without loss of generality we assume that each $\varphi_k$ has positive leading coefficient. This choice ensures uniqueness of the functions $\varphi_k$.

\begin{remark}\label{nointformularem}
In the case $\alpha=0$, where the norm can be expressed in terms of integrals,
\[\|f\|^2_{0}=\lim_{r\to 1}\frac{1}{2\pi}\int_{-\pi}^{\pi}|f(re^{i\theta})|^2d\theta,\]
the space $D_{0,f}$ is simply the weighted Hardy space $H^2(\mu)$ with
$d\mu=|f|^2d\theta$.

Similarly, when $\alpha=-1$, the space $D_{-1,f}$ is a weighted Bergman space with norm given by
\[\|g\|^2_{-1,f}=\int_{\DD}|g(z)|^2d\mu(z),\]
with $d\mu=|f|^2dA$, where $dA$ denotes normalized area measure.

For other choices of $\alpha$, however, equivalent expressions for the norm of $g\in D_{\alpha}$ are
given in terms of the integrals
\[|g(0)|^2+\int_{\DD}|g'(z)|^2(1-|z|^2)^{1-\alpha}dA(z),\]
and the presence of a derivative means that it is not possible, in general, to write $\langle g, h\rangle_{\alpha,f}$
in terms of weighted $L^2$-type inner products in a simple way.
\end{remark}

The optimal approximant $p_n$ minimizes $\|pf - 1\|_{\alpha}$ over the space of polynomials $ p \in \Pol_{n}$, and therefore is the projection of $1$ onto $f\cdot\Pol_{n}$. Hence, $p_nf$ can be expressed by its Fourier coefficients in the basis $\varphi_k f$ as follows:
\[ (p_n f)(z) = \sum_{k=0}^{n} \langle 1, \varphi_k f \rangle_{\alpha} \varphi_k(z) f(z).\]
Eliminating $f$ from both sides of the expression gives
\[ p_n(z) = \sum_{k=0}^{n} \langle 1, \varphi_k f \rangle_{\alpha} \varphi_k(z).\]
Notice that by the definition of the inner product \eqref{ipseries}, in all the $D_{\alpha}$ spaces we have
\begin{align*}
\langle 1, \varphi_k f \rangle_{\alpha}
&= \overline{\varphi_k(0)  f(0)}.
\end{align*}
We have thus proved the following.

\begin{proposition}\label{orthopolrepresentation}
    Let $\alpha \in \RR$ and $f \in D_{\alpha}$.  For integers $k \geq 0,$ let $\varphi_k$ be the orthogonal polynomials for the weighted space $D_{\alpha, f}$.  Let $p_n$ be the optimal approximants to $1/f$.  Then
    \[ p_n(z) = \overline{f(0)} \sum_{k=0}^n \overline{\varphi_k(0)} \varphi_k(z).\]
\end{proposition}
\begin{remark}
Another way to read this expression is as a way to recover the orthogonal polynomials $\varphi_k$ from the difference between optimal approximants and their values at $0$: provided $\varphi_n(0) \neq 0$, we have $$ \varphi_n(z) = \frac{p_n(z) - p_{n-1}(z)}{\overline{\varphi_n(0)}\overline{f(0)}}.$$
We can even recover the modulus of the value at the origin,
\[|\varphi_n(0)|= \sqrt{ \frac{p_n(0) - p_{n-1}(0)}{\overline{f(0)}}}.\]
\end{remark}
When $\alpha = 0$, Proposition \ref{orthopolrepresentation} quickly leads to insights into the nature of zero sets of optimal approximants.

\begin{theorem}\label{t-noroots}
    Let $f \in H^2$, and let $p_n$ be the optimal approximant to $1/f$.  Then $p_n$ has no zeros inside the closed disk.
\end{theorem}

\begin{proof}
For $f \in H^2$ given, define the positive measure $d \mu(\theta) = |f(e^{i \theta})|^2 d \theta$ on the circle, and consider the weighted Hardy space $H^2(\mu)$ of analytic functions $g$ in the disk that satisfy
$$\int_0^{2 \pi} |g(e^{i\theta})|^2 d \mu(\theta) < \infty.$$ Let $\varphi_k$ be the orthogonal polynomials for the space $H^2(\mu)$, normalized so that the leading coefficient $A_k$ of $\varphi_k$ is positive.   Now define
\begin{equation}\label{phi*}
    \varphi_n^{*}(z) : = z^n \overline{\varphi_n}(1/z),
\end{equation}
 where the polynomial $\overline{\varphi_n}$ is obtained by taking conjugates of the coefficients of $\varphi_n$.    Notice that
if $\varphi_n(z) = A_n z^ n + \sum_{j=0}^{n-1} a_j z^j$ then $ \varphi_n^{*}(z) = A_n + \sum_{j=0}^{n-1} \overline{a_j} z^{n-j}$.
Now it is well-known from the theory of orthogonal polynomials (see for example \cite[Chapter 1]{Ger} or \cite[Chapter 1]{SimonBook}) that
\begin{equation}\label{ortho}
     \varphi_n^{*}(z) = \frac{1}{A_n} \, \sum_{k=0}^{n} \overline{\varphi_k(0)} \varphi_k(z).
\end{equation}
Therefore by Proposition \ref{orthopolrepresentation}, the optimal approximants $p_n$ are multiples of the $n-$th ``reflected" orthogonal polynomial:
\begin{equation*}
    p_n(z) = \overline{f(0)} \,  A_n \, \varphi_n^*(z).
\end{equation*}
Therefore the zeros of $p_n$ are the same as the zeros of
$\varphi_n^*.$ Moreover it is clear from \eqref{phi*} that $z$ is a
zero of $\varphi_n^*$ if and only if $1/\bar{z}$ is a zero of
$\varphi_n$.  Finally, again from the theory of orthogonal
polynomials, it is well-known that their zeros lie inside the open unit
disk (see \cite[Chapter 1]{Ger}), and therefore, the zeros of $p_n$
lie outside the closed unit disk, as desired.
\end{proof}
In Section \ref{s-RK}, we give a different argument extendable to all values of $\alpha$.

In \cite{BCLSS13}, optimal approximants were used to study cyclic vectors, but it is instructive to see what happens also in the case when $f$ is not cyclic.
\begin{example}[(Blaschke factor in the Hardy space)]\label{blaschkeex}
Let $\lambda \in \DD\setminus \{0\}$, and consider the case of a single Blaschke factor
\[f_{\lambda}(z) = \frac{\lambda - z}{1 - \overline{\lambda} z},\]
a function that is certainly not cyclic in $H^2$ (or in
any $D_{\alpha}$ for that matter). First note that
$|f_{\lambda}(e^{it})|=1$ implies $\|p_n f_\lambda-1\|_{H^2}= \|p_n-
1/f_\lambda\|_{L^2}$, and hence the orthogonal polynomials are $\varphi_k
= z^k$, $k\geq 0$. The optimal approximants are given by
\[ p_n(z) = \sum_{k=0}^{n} \langle 1/f_\lambda, z^k \rangle_2\, z^k.\]
Note that $1/f_\lambda$ is not analytic in $\mathbb D$, but is analytic in $\mathbb C\setminus \overline{\mathbb{D}}$. A calculation shows that
\[
1/f_\lambda = \overline{\lambda} + (|\lambda|^2-1) z^{-1} + \cdots
\]
Therefore, in $L^2(\mathbb{T})$, we obtain the coefficients
\[
 \langle 1/f_\lambda, z^k \rangle_2
=
\left\{
\begin{array}{ll}
\overline{\lambda} & k=0\\
0 & k\in \mathbb{N}_+
\end{array}.
\right.
\]

In conclusion, the $n$th optimal approximant is given by $p_n(z) = \overline{\lambda}$ for all $n$, and so is non-vanishing in the closed disk as guaranteed by Theorem \ref{t-noroots}. It is not hard to verify
\begin{multline*}
\dist_{H^2}(1, f_\lambda \mathcal{P}_n)
= \int_\mathbb{T} |p_n - 1/f_\lambda|^2 dm
=\int_\mathbb{T} \left|\overline{\lambda} - \frac{1 - \overline{\lambda} z}{\lambda - z}\right|^2 dm
= 1-|\lambda|^2.
\end{multline*}
Phrased differently, we have $\dist_{H^2}(1, f_{\lambda}\cdot \mathcal{P}_n)=1-f(0)p_n(0).$
In particular, we recover what we already know: $f_{\lambda}$ is only cyclic when $\lambda=1$ (and $f$ is interpreted as being constant).

Both of these observations (non-vanishing of $p_n$, distance formula) will be discussed further in the next sections.
\end{example}

The formula for the $n$-th reflected orthogonal polynomial expressed in \eqref{ortho} relies heavily on the fact that $f \in H^2$ and that the orthogonal polynomials in this context come from a measure defined on the circle. As was explained in Remark \ref{nointformularem}, no such formula expressing a direct relationship between the $n$-th optimal approximant and the $n$-th reflected orthogonal polynomial holds for measures defined on the disk, and so for Dirichlet spaces $D_{\alpha}$ where
$\alpha \neq 0$, such as the Bergman space for example, one must search for different tools.
It turns out that the language of reproducing kernels is useful in this context.

\section{Reproducing kernels and zeros of optimal approximants}\label{s-RK}

Let us return to the case of an arbitrary $\alpha \in \RR$, fix $f \in D_{\alpha}$, let $n$ be a non-negative integer, and let $\varphi_k$ be the orthogonal polynomials that form a basis for $f\cdot\Pol_{n},$ for $ 0 \leq k \leq n$.

In general, if $k(\cdot, w)$ is the reproducing kernel function at $w$ in a reproducing kernel Hilbert space $H$, then
\[k(z,w)=\sum_{k=0}^{\infty}\overline{\psi_k(w)}\psi_k(z)\]
for any orthonormal basis $(\psi_k)$, see \cite{AgMcCBook}.
Inspecting the relation \eqref{weightedspacedef} now leads to the
conclusion that the function
\begin{equation}\label{repkernel}
    K_n(z,w): = \sum_{k=0}^n \overline{\varphi_k(w) f(w)} \varphi_k(z) f(z)
\end{equation}
is the reproducing kernel for the space $f\cdot\Pol_{n}$. Recall that the reproducing kernel $K_n$
of the subspace $f\cdot \mathcal{P}_n\subset D_{\alpha}$ is characterized by the property that, for every $g\in f\cdot \mathcal{P}_n$,
\[g(w)=\langle g, K_n(\cdot, w)\rangle_{\alpha}, \quad w\in \DD.\]

Therefore, by Proposition \ref{orthopolrepresentation}, the optimal approximants to $1/f $ are related to these reproducing kernels as follows:
\begin{equation}\label{optrep}
    K_n(z,0) = p_n(z) f(z).
\end{equation}
One consequence of this fact is the following proposition, whose proof is standard and is included for completeness.
\begin{proposition}\label{extremal}
    The function $K_n(z,0)/\sqrt{K_n(0,0)}$ is extremal for the problem of finding
    $$ \sup \{ |g(0)|: g \in f\cdot\Pol_{n}, \, \|g\|_{\alpha} \leq 1 \},$$ and thus the supremum is equal to $\sqrt{K_n(0,0)}.$
\end{proposition}


\begin{proof}
First note that $\|K_n(\cdot, 0)\|^2_{\alpha} = \langle K_n(\cdot, 0), K_n(\cdot, 0) \rangle_{\alpha} = K_n(0,0),$ by the reproducing property of $ K_n(\cdot, 0).$ Now let $g$ be any function in $f\cdot\Pol_{n}$ such that
$ \|g\|_{\alpha} \leq 1.$  Then
\[ |g(0)| = | \langle g,  K_n(\cdot, 0) \rangle_{\alpha} | \leq \|g \|_{\alpha} \, \|K_n(\cdot, 0)\|_{\alpha} \leq \sqrt{K_n(0,0)}.\]
Choosing $g(z) = K_n(z,0)/\sqrt{K_n(0,0)}$ gives that $\|g\|_{\alpha}= 1$ and $g(0) = \sqrt{K_n(0,0)},$ and thus $g$ is a solution to the extremal problem stated in the proposition, as required.
\end{proof}

Expressing the optimal approximants in terms of these kernels allows us to prove our main result concerning zeros of optimal approximants.
\begin{theorem}\label{norootsallalpha}
Let $\alpha \in \RR,$ let $f \in D_{\alpha}$ have $f(0)\neq 0$, and let $p_n$ be the optimal approximant to $1/f$ of degree $n$. Then
\begin{itemize}
\item if $\alpha \geq 0$, all the zeros of the optimal approximants lie outside the closed unit disk;
\item if $\alpha < 0,$ the zeros lie outside the closed disk $\overline{D}(0,2^{\alpha/2})$.
\end{itemize}
\end{theorem}
It is clear that the kernels $K_n$ vanish at all the zeros of $f$. Borrowing terminology from Bergman space theory, we say that any $\lambda \in \CC$ such that
$K_n(\lambda,0)=0$ but $f(\lambda)\neq 0$ is an {\it extraneous zero}. Theorem \ref{norootsallalpha} can then be rephrased by saying that the reproducing kernels $K_n(\cdot, 0)$ have no extraneous zeros in $\overline{\DD}$ when $\alpha\geq 0$, and no extraneous zeros in $\overline{D}(0,2^{\alpha/2})$ when $\alpha<0$.

\begin{proof}
Let $k_n(z) := K_n(z,0) = p_n(z) f(z)$ be the reproducing kernel at $0$ for the space $f\cdot\Pol_{n}$, and suppose $\lambda$ is an extraneous zero of $k_n$.  Then
\[ k_n(z) = (z-\lambda) q(z) f(z),\]
where $q$ is a polynomial of degree at most $n-1$.  Therefore
$$\lambda q(z) f(z) = z q(z) f(z) - k_n(z).$$
Notice that since $z q \in f\cdot\Pol_{n}$ and vanishes at $0$, while $k_n$ reproduces at $0$,
\[0=(zqf)(0)=\langle zqf, k_n\rangle_{\alpha},\]
the two functions $zqf$ and $k_n$ are orthogonal. It follows that
\begin{equation}\label{aqf}
|\lambda|^2 \|qf\|_{\alpha}^2 = \|zqf\|_{\alpha}^2 + \|k_n\|_{\alpha}^2.
\end{equation}
For any function $F(z) = \sum_{n=0}^{\infty} a_n z^n$ in $D_{\alpha}$, we have
$$\|F\|_{\alpha}^2 = \sum_{n=0}^{\infty} (n+1)^{\alpha} |a_n|^2$$ while
$$\|zF\|_{\alpha}^2 = \sum_{n=0}^{\infty} (n+2)^{\alpha} |a_n|^2 = \sum_{n=0}^{\infty} \left( \frac{n+2}{n+1} \right)^{\alpha} (n+1)^{\alpha} |a_n|^2.$$

It is clear that
\[1 \leq \frac{n+2}{n+1} \leq 2.\]
Hence, if $\alpha \geq 0,$ then $\|zF\|_{\alpha}^2 \geq \|F\|_{\alpha}^2$, while if $\alpha < 0$, we obtain $\|zF\|_{\alpha}^2 \geq 2^{\alpha}\|F\|_{\alpha}^2$.
Applying these estimates to $F = q f$ in \eqref{aqf}, we obtain, for $\alpha \geq 0,$ that
$$ ( |\lambda|^2 - 1 ) \|qf\|_{\alpha}^2 \geq \|k_n\|_{\alpha}^2 > 0,$$ which implies that $|\lambda| > 1$,  as claimed. For $\alpha < 0$, it follows that
$$ ( |\lambda|^2 - 2^{\alpha} ) \|qf\|_{\alpha}^2 \geq \|k_n\|_{\alpha}^2 > 0,$$ which implies that $|\lambda| > 2^{\alpha/2},$
as desired.
\end{proof}
A few remarks are in order.
\begin{remark}
Note that in the case of the Bergman space another way to see the relationship between the norm of a function and the norm of its multiplication by $z$ is to recall that the function $\sqrt{2}\,z$ is the so-called ``contractive divisor" at $0$, and thus is an expansive multiplier (see \cite{HKZ00} or \cite{DS04}).  Therefore one has $ \|\sqrt{2} z F\|_{-1} \geq \|  F \|_{-1}$, which is equivalent to the desired inequality.  The same remark applies to $D_{\alpha}$ in the range $\alpha\in [-2,0]$.  Moreover, it is straightforward to show (see, e.g., \cite{CW13}) that when $n \rightarrow \infty$, $k_n (z) \rightarrow k(z,0) $, where $k(z,w)$ is the reproducing kernel in the weighted space $D_{\alpha,f}$, when $f$ is sufficiently nice up to the boundary of the disk. As is known (\cite{DKS96}), $k(z,0)$ has no extraneous zeros. Thus, for $\alpha \in [-2,0)$, the zeros of $k_n$ are all eventually ``pushed out" of the unit disk when $n \rightarrow \infty.$
\end{remark}

\begin{remark}
The proof of Theorem \ref{norootsallalpha} is similar to a well-known proof (due to Landau, according to \cite{SimonBook}) about the location of the
the zeros of orthogonal polynomials in a fairly general setting. For example, suppose $\mu$ is any measure on the unit disk $\DD$ and let $\varphi_n $  be the orthogonal polynomial of degree $n$ with respect to $\mu$, normalized for instance by requiring its leading coefficient to be positive. Then
$$ \int_{\DD} \varphi_n(z) \varphi_m(z) d \mu(z) = 0$$ if $n \neq m$, and so $\varphi_n$ is orthogonal to any polynomial of degree strictly less than $n$.  Now if $\varphi_n(\lambda) = 0$, we can write $\varphi_n(z) = (\lambda - z) q(z),$ where $q$ is a polynomial of degree $n-1$.  Then
$ z q(z) = \lambda q(z) - \varphi_n(z)$, and therefore
$$ \|z q\|^2 = |\lambda|^2 \|q\|^2 + \|\varphi_n\|^2.$$ Since $z \in \DD$, $\|z q \| \leq \|q\|,$ and therefore we obtain that
$$ (1 - |\lambda|^2) \|q\|^2 \geq \|\varphi_n\|^2 > 0,$$ which implies that $|\lambda| < 1$. In fact, one could refine the estimate further based on the support of $\mu$, for instance if $\mu$ were an atomic measure, since
$$ \|z q\|^2 \leq \max \{ |z| : z \in \mathrm{supp}(\mu) \} \cdot \|q\|^2,$$ one would obtain that $|\lambda| \leq \max\{ |z| : z \in \mathrm{supp}(\mu) \}$.
\end{remark}

\begin{remark}
We do not know whether the radius $2^{\alpha/2}$ is optimal, that is, whether there are examples of
optimal approximants, associated with functions in $D_{\alpha}$ with $\alpha$ negative, that vanish at points $\lambda$ with modulus arbitrarily close to $2^{\alpha/2}$.

However, in Section \ref{s-Extraneous} we present examples of functions that lead to extraneous zeros located at $\lambda\approx 0.88$ when $\alpha=-2$, and at $\lambda\approx 0.98$ when $\alpha=-1$. Hence
$\overline{D}(0,2^{\alpha/2})$ cannot be replaced by $\overline{\DD}$ in the second part of Theorem \ref{norootsallalpha}.
\end{remark}

\section{Conditions for cyclicity}\label{s-kernelscyclicity}
We can use the reproducing kernels to give equivalent criteria for the cyclicity of $f$ in terms of the pointwise convergence of kernels at a single point, the origin.

\begin{theorem}\label{cyclicitythm} Let $f \in D_{\alpha}$, $f(0)\neq 0$, and let $p_n$ be the optimal approximant to $1/f$ of degree $n$.  Let $\varphi_k$ be the orthogonal polynomials for the weighted space $D_{\alpha, f}$.
    The following are equivalent.
    \begin{enumerate}
    \item $f$ is cyclic.
    \item $p_n(0)$ converges to $1/f(0)$ as $n \rightarrow \infty$.
    \item $\sum_{k=0}^{\infty} |\varphi_k(0)|^2 = 1/|f(0)|^2$.
    \end{enumerate}
\end{theorem}

In the next Section we will extend this theorem in the case when $\alpha=0$, by including an additional equivalent condition.

\begin{proof}
We first record some observations. By \eqref{optrep} we find:
\begin{align}\label{e-IFF}
f \text{ is cyclic}\quad\Leftrightarrow\quad \|K_n(\,\cdot\,,0)- 1\|_{\alpha}\to0.
\end{align}
From Equation \eqref{repkernel}, we obtain
\begin{align}\label{e-A}
\|K_n(\,\cdot\, ,0)\|^2_{\alpha} = K_n(0,0) = |f(0)|^2 \sum_{k=0}^n |\varphi_k(0)|^2.
\end{align}
Now, for any function $h$ in $D_\alpha$ we have the orthogonal decomposition of the norm as
\begin{align}\label{e-D}
\|h\|^2_{\alpha} = |h(0)|^2 + \|h - h(0)\|^2_{\alpha}.
\end{align}

Now we show (1)$\,\Rightarrow\,$(2)$\,\Rightarrow\,$(3)$\,\Rightarrow\,$(1).

Assume (1). Then (2) follows from pointwise convergence of $p_n$ to $1/f$. By \eqref{optrep} we obtain $K_n(0,0)\to 1$, which implies item (3) by virtue of \eqref{e-A}.

It remains to argue that (1) follows from (3). Assume that (3) holds. Then by \eqref{e-A} as $n\to\infty$:
\begin{align}\label{e-B1}
K_n(0,0)&\to 1,\text{ and}\\
\|K_n(\,\cdot\, ,0)\|&\to 1.\label{e-C}
\end{align}

Equation \eqref{e-D} applied to $h=K_n(\,\cdot\, ,0)$ yields
\[
\|K_n(\,\cdot\, ,0)\|^2_{\alpha} = |K_n(0 ,0)|^2 + \|K_n(\,\cdot\, ,0)-K_n(0 ,0)\|^2_{\alpha}.
\]
In view of \eqref{e-B1} and \eqref{e-C} we learn $$\|K_n(\,\cdot\,
,0) - K_n(0,0)\|_{\alpha}\to 0.$$

Since monomials are orthogonal in $D_\alpha$,
\begin{equation}
\|K_n(\,\cdot\,,0) -1\|^2_{\alpha} = \|K_n(\,\cdot\,,0) - K_n(0,0)\|^2_{\alpha}+
|K_n(0,0)-1|^2
\end{equation}
and this, again together with \eqref{e-B1}, informs us that $\|K_n(\,\cdot\,,0) -1\|_{\alpha}\to 0$. Item (1) now follows from \eqref{e-IFF}.
\end{proof}

A further equivalent criterion can be formulated by relating the distance $\mathrm{dist}_{D_{\alpha}}(1, f \cdot \mathcal{P}_n)$ to
the values of $p_n$ at the origin, as in Example \ref{blaschkeex}. This is actually part of a general statement contained in a
classical result of Gram, which we phrase here in our terminology,
although it applies in any Hilbert space, and for distances
involving more general finite-dimensional subspaces.

For $f\in D_{\alpha}$ fixed, consider the matrix
\[M_n=\left(M_{jk}\right)_{j,k=0}^{n}=\left(\langle z^j f, z^k
f\rangle_{\alpha} \right)_{j,k=0}^{n}\] and denote its lower right $n$-dimensional minor by $\hat{M}_n$.

\begin{lemma}[(Gram's Lemma)]
Let $f \in D_\alpha$. Then $d_n=\mathrm{dist}_{D_{\alpha}}(1,f\cdot
\mathcal{P}_n)$ satisfies \begin{equation} d_n^2 = 1 -
p_n(0)f(0)\end{equation} where $p_n$ is the $n$th optimal
approximant to $1/f$. Moreover, $p_n(0)$ is given as
\begin{equation}\label{gram1} p_n(0)= \overline{f(0)}
\det{\hat{M}_n}/\det{M_n}.\end{equation}
\end{lemma}
\begin{proof}
First, notice that, since the orthogonal projection of $1$ onto
$\mathcal{P}_n f$ of $1$ is $p_nf$, we have
\begin{equation}\label{orthogonal1}
d^2_n = \left\langle p_n f - 1 , p_n f - 1\right\rangle_{\alpha}.
\end{equation}
Since $p_nf - 1$ is orthogonal to all
functions of the form $q f$ where $q$ is a polynomial of degree less
or equal to $n$, we obtain that
\begin{equation}\label{orthogonal2}
d^2_n = \left\langle p_n f - 1 , - 1\right\rangle_{\alpha} = 1-p_n(0)f(0).
\end{equation}
The matrix $M_n$ satisfies $M_n c = b$ where $c$ are the
coefficients of $p_n$ and $b_k = \langle1,z^kf\rangle_{\alpha}$. These equations together with
\eqref{orthogonal2} form a system of equations with unknowns
$d^2_n$ and $c$. Using Cramer's rule to solve for $d^2_n$ gives the
determinant identity \eqref{gram1}.
\end{proof}

Gathering everything we have obtained so far, and including Gram's Lemma,
we obtain the following.
\begin{corollary}\label{cyclicitycor}
Let $f \in D_\alpha$ satisfy $f(0)\neq 0$ and let $p_n$, $M_n$, and $K_n$ be as above. Then the following quantities are all equal:
\begin{itemize}
\item[(a)] $\dist^2_{D_{\alpha}}(1, f\cdot \mathcal{P}_n)$
\item[(b)] $\|p_nf-1\|_{\alpha}^2$
\item[(c)] $1 -p_n(0)f(0)$
\item[(d)] $1 - ((M_n)^{-1})_{0,0} |f(0)|^2$
\item[(e)] $1 - \sum_{k=0}^n |\varphi_k(0)|^2 |f(0)|^2$
\item[(f)] $1 - K_n(0,0)$
\end{itemize}
If, moreover $\alpha=0$ and the degree of $p_n$ is equal to $n$,
then all the above are also equal to $1-\hat{\varphi}_n^2(n)
|f(0)|^2$.
\end{corollary}

Hence, $f$ is cyclic if any (hence all) of these quantities tend to zero with $n$.
Since the distance in $(\mathrm{a})$ above is
always a number between 0 and 1 and converges (as $n$ goes to
$\infty$) to the distance from $1$ to $[f]$, and the numbers in item $(\mathrm{e})$ are non-increasing, all the other
quantities converge in the interval $[0,1]$. In particular, a
function $f$ is cyclic if and only if the kernel of the invariant subspace generated by $f$, $K_{[f]}$,
satifies $K_{[f]}(0,0)=1$.

\begin{remark}[(A formula of M\McC Carthy)]

We point out a connection with an observation of M\McC Carthy, see \cite[Theorem 3.4]{McC94}. Under the assumption that $f\in H^{\infty}$ is cyclic in the Bergman space ($\alpha = -1$), he provides a closed formula for the reproducing kernel $K$ of the closure of the polynomials with respect to $\|\cdot\|_{-1,f}$.  His result generalizes to $D_{\alpha}$ with
$\alpha<0$, and yields
\[K(z,w)=\frac{1}{\overline{f(w)}f(z)}\frac{1}{(1-\bar{w}z)^{1-\alpha}}, \quad z,w\in \DD.\]
This is in effect a rescaling of the reproducing kernel of $D_{\alpha}$; see \cite[Chapter 2.6]{AgMcCBook} for a discussion of this notion.
\end{remark}

\section{Toeplitz matrices and Levinson algorithm} \label{s-Toeplitz}

In view of \eqref{approxlinearsyst} and Corollary \ref{cyclicitycor}, it is of interest to consider different algorithms for inverting the
matrices $M_n$.

Multiplication by $z^l$ is an isometry on $H^2$. Therefore, in the
case of the Hardy space, the matrices $M= (M_{k,l})$ appearing in
the determination of the optimal approximant have the property that
$M_{k,l}= M_{k-l,0}$. In other words, the entries $M_{k,l}$ only
depend on the distance to the diagonal. A matrix with this property
is called a \emph{Toeplitz matrix}. We can use this structure of the
matrices to extend our results in Theorem \ref{cyclicitythm}.  A number of algorithms have been developed specifically for inverting Toeplitz matrices. In \cite[p.~7--13]{HR11} several methods are mentioned, based either on Levinson's \cite{Lev} or Schur's \cite{Sch} algorithms.

\begin{theorem}\label{charact}
Let $f \in H^2$ be such that $f(0) \neq 0$, and denote $c_{k,n}$ the $k$-th coefficient of the $n$-th optimal approximant to $1/f$. Then $f$ is outer if and only if $$\prod_{n=0}^{\infty} \left(1- \frac{|c_{n+1,n+1}|^2}{|c_{0,n+1}|^2}\right) =\overline{f(0)}/ \|f\|^2.$$
\end{theorem}

\begin{remark}\label{zerocharrem}
Notice that, in the notation for orthogonal polynomials used in the
previous sections, $c_{n,n}/c_{0,n} = \overline{\varphi_{n}(0)}/
\hat{\varphi}_n(n)$, and this quotient is also the product of the numbers
$\overline{z_k^{-1}}$, where $z_k$ varies over all zeros of $p_n$, or
alternatively, the product of the zeros of the orthogonal polynomial
$\varphi_n$. In particular, in the Hardy space, cyclicity can be
characterized exclusively in terms of the zeros of optimal
approximants or in terms of those of orthogonal polynomials. Theorem \ref{charact} is, in a sense, a
qualitative optimal approximant version of the known characterization of outer functions as those satisfying
$ \log |f(0)| = \frac{1}{2 \pi} \log |f| d \theta.$
It would be of great interest to know whether a version of Theorem \ref{charact}
also holds in other spaces.
\end{remark}

\begin{proof}

Without loss of generality we assume $f(0)=1$ (otherwise divide $f$ by $f(0)$). As was explained in the introduction, the coefficients $c = (c_0, c_1, \hdots, c_n)$ of the optimal approximant of order $n$ are given by the linear system
$$M\vec{c}=\vec{e}_0 \quad\text{where}\quad  e_0 = (1,\vec{0})
\quad\text{and}\quad M_{k,l} = \left\langle z^kf, z^l f\right\rangle.$$
By virtue of the existence and uniqueness of the minimization problem, the matrix $M$ is invertible.  Our objective is to obtain the coefficients by taking $$\vec{c}=M^{-1} \vec{e}_0.$$

Now we will use the fact that $M$ is a Toeplitz matrix, and apply the Levinson algorithm. As our matrix is in fact Hermitian, we can apply a slightly simplified version of this procedure. The algorithm is based on the fact that all information of the matrix is contained in two columns (when $M$ is Hermitian, in one column).

The solution is as follows: If $\{c_{k,n}\}_{k=0}^k$ are the coefficients of the $n$th-degree optimal approximant, then the coefficients $\{c_{k,n+1}\}_{k=0}^{n+1}$ of the optimal approximant of degree $n+1$ can be obtained from those previous coefficients:
\begin{equation}\label{eqn199}
c_{k,n+1}= \frac{1}{1-|\Gamma_n|^2} \left(c_{k,n} - \Gamma_n \overline{c_{n+1-k,n}}\right),
\end{equation}
where $$\Gamma_n = \sum_{k=0}^n c_{n-k,n} \left\langle z^{k+1}f,f\right\rangle.$$

Since $f(0)=1$, the numbers $c_{0,n}$ are always real. From the expression above, we can then obtain
\begin{equation}\label{eqn200}
c_{n+1,n+1} = \frac{-\Gamma_n}{1-|\Gamma_n|^2} c_{0,n} = - \Gamma_n c_{0,n+1}.
\end{equation}

Finally, this gives us
\begin{equation}\label{eqn201}
\Gamma_n= -\frac{c_{n+1,n+1}}{c_{0,n+1}}.
\end{equation}

From \eqref{eqn199} we can recursively recover the value of $p_n(0)$:
\begin{equation}\label{eqn210}
p_{n+1}(0) = \frac{c_{0,n}}{1-|\Gamma_n|^2} =\cdots= \frac{c_{0,0}}{\prod_{k=0}^n (1-|\Gamma_k|^2)},
\end{equation} where $\Gamma_n$ is defined as in \eqref{eqn201}. The value of $c_{0,0}$ can be recovered from the corresponding equation $\|f\|^2c_{0,0}=1$ and thus, \eqref{eqn210} becomes

\begin{equation}\label{eqn211}
p_{n+1}(0) = \frac{1}{\|f\|^2 \prod_{k=0}^n (1-|\Gamma_k|^2)}.
\end{equation}

 Being outer is equivalent to cyclicity in Hardy space by the classical theorem of Beurling. By Theorem \ref{cyclicitythm}, this will happen if and only if $p_n(0)$ tends to 1 as $n$ tends to infinity, which happens if and only if
$$\prod_{k=0}^{\infty} (1-|\Gamma_k|^2) = 1/\|f\|^2.$$

Using \eqref{eqn201} again to translate the value of $\Gamma_k$, we obtain the desired result.
\end{proof}

\begin{example}[(Optimal approximants to $1/(1-z)$ revisited)]
We illustrate how the Levinson algorithm can be exploited for our purposes by
using it to re-derive the optimal approximants for the basic example $f=1-z$. In this case, the main advantage is that $\Gamma_k$ is very simple:
$$\Gamma_n= c_{n,n}, \quad n\in \NN.$$

So to verify that Ces\`aro polynomials, which correspond to the choice $c_{k,n} = 1- (k+1)/(n+2)$, are optimal, we just need to check that they satisfy the recursive formula and the initial condition for degree $0$.

Hence, we want to show that $$c_{k,n+1}= \frac{1}{1-|c_{n,n}|^2} \left(c_{k,n}+c_{n,n}\overline{c_{n+1-k,n}}\right).$$
Evaluating both sides of the formula reduces our task to checking that
$$\frac{n+2-k}{n+3} = \frac{(n+2)(n+1-k)+k}{(n+3)(n+1)}.$$
Multiplying both sides by $(n+3)(n+1)$, we obtain that
\[(n+1)(n+2-k)=n^2+3n+2-nk-k=(n+2)(n+1-k)+k,\]
and therefore the proposed polynomials are optimal as claimed.
\end{example}

\begin{remark}
A particularly simple case is that of $f$ an inner function in the Hardy space. Then, in the notation of the proof of the previous Theorem, $\Gamma_k=0$ for all $k>0$, and the polynomials do not change with $k$. This means that either the optimal norm converges to $0$ with $p_n$ all being equal to a constant, or it does not converge. In fact, when $f=I F$, where $I$ is inner and $F$, outer, the elements of the system \eqref{approxlinearsyst} depend exclusively on $F$. That is, for $f \in H^2$, the optimal approximants depend only on the outer part of $f$.
Another immediate consequence is that for any outer nonconstant function $f$, there is some $t \geq 1$ such that $\left\langle
z^t f, f\right\rangle \neq 0$. 
In other words, inner functions are \emph{characterized} by having optimal approximants of all
degrees equal to a constant.

\end{remark}

\section{Extraneous zeros}\label{s-Extraneous}
We now return to the zero sets of optimal approximants, with a view towards determining the location of zeros analytically. Our first result states that the roots of $p_n$ can be expressed in terms of certain inner products.
\begin{lemma}\label{zeroreprlemma}
Let $f \in D_\alpha$ have $f(0)\neq 0$,and let $\mathcal{Z}(p_n)=\{z_1,\ldots, z_n\}$.

Then $z_1,\ldots, z_n$ are given by the unique-up-to-permutation solution to the system of equations
\begin{equation}\label{eq4}
z_m = \frac{\left\|zf \Pi_{j=1, j \neq m}^n (z-z_j)\right\|^2}{\left\langle f\Pi_{j=1, j \neq m}^n (z-z_j), zf \Pi_{j=1, j \neq m}^n (z-z_j) \right\rangle}
\end{equation}
for $m=1,...,n$.

In particular, the zero of $p_1$, the first order approximant, is given by
\begin{equation}\label{zeroformula}
z_1=\frac{\|zf\|^2_{\alpha}}{\left\langle f,zf\right\rangle_{\alpha}}.
\end{equation}
\end{lemma}
If some zero $z_m$ is repeated, the solution is still unique but we count multiplicity. Note also that if $f$ is a cyclic function but not a constant or a rational function, then there is an infinite subsequence $\{n_k\}_{k \in \NN}$ such that $\deg (p_{n_k})= n_k$; if not, $\|p_nf-1\|_{\alpha}$ cannot tend to $0$ as $n \to \infty$.

\begin{proof}
The first-order approximant $p_1=c_0+c_1z$ to $1/f$
is obtained by solving the system of equations
\begin{equation}\label{firstrow}
\|f\|^2_{\alpha} c_0 +\left\langle z f,f\right\rangle_{\alpha} c_1 = \overline{f(0)}
\end{equation}
and
\begin{equation}\label{secondrow}
\left\langle f,zf\right\rangle_{\alpha} c_0 + \|zf\|^2_{\alpha} c_1 = 0.
\end{equation}
Suppose $c_1\neq 1$ (otherwise interpret $z_0$ as being equal to infinity). Then
$p(z_1)= 0$ is equivalent to
\begin{equation*}\label{onezero}
z_1= -c_0/c_1.
\end{equation*}
and by \eqref{secondrow} then, we obtain
\begin{equation*}\label{ratioform}
z_1= \frac{\|zf\|^2_{\alpha}}{\left\langle f,zf\right\rangle_{\alpha}}.
\end{equation*}

Next, we note that \eqref{zeroformula} can be expressed as the orthogonality condition
\begin{equation}\label{orthoform}
\left\langle (z-z_1)f,zf\right\rangle = 0.
\end{equation}
To prove the lemma for the optimal approximant of $1/f$ of any degree, it is enough to apply Equation \eqref{orthoform} to a function $g$ that is the product of $f$ with a polynomial of degree $n-1$:
\begin{equation*}
\left\langle f\Pi_{j=1}^n (z-z_j), zf \Pi_{j \neq i, j=1}^n (z-z_j) \right\rangle=0.
\end{equation*}
If the optimal polynomial to invert $f$ has $n$ zeros, each of them has to satisfy a corresponding orthogonality condition for a different $g$. Moreover, the fact that we multiply the polynomial by a constant does not affect the orthogonality condition, and the zeros are determined exactly by those orthogonality conditions. Since the polynomials are unique, the zeros are also uniquely determined.
\end{proof}

\begin{figure}
\includegraphics[width=0.2\textwidth]{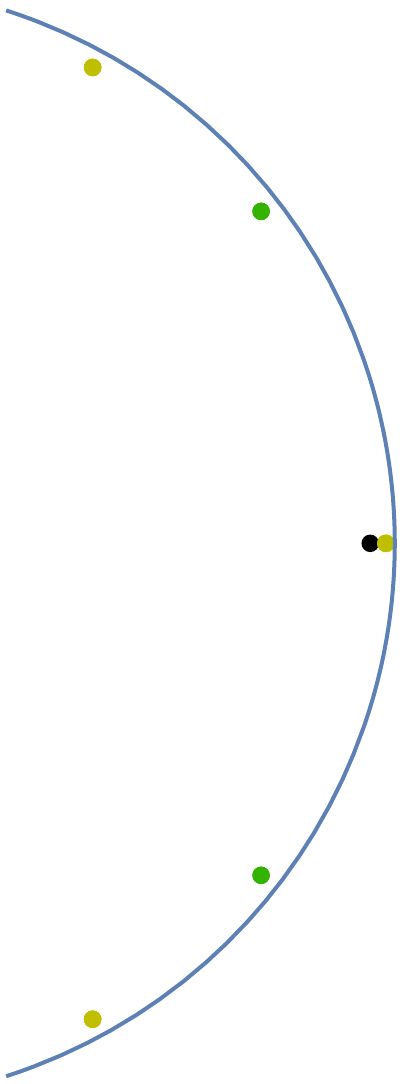}
\caption{Zeros of the optimal approximants $p_1$ (black), $p_2$ (green), and $p_3$ (gold) associated with $(1+z)^3$ in the weighted Bergman space $D_{-2}$.}
\label{extraneousfig}
\end{figure}
We shall now use Lemma \ref{zeroreprlemma} to show that optimal approximants to $1/f$ have zeros in $\DD$ for judiciously chosen $f$, or in other words, that the associated kernels $K_n$ have extraneous zeros.
We present two families of examples, one that is completely elementary, and one that requires more work but has the advantage of producing
extraneous zeros in the disk for the classical Bergman space.

\begin{example}[(Extraneous zeros in weighted Bergman spaces)]\label{smallex}
We begin by treating the spaces $D_{\alpha}$ with $\alpha<-1$. An equivalent norm for $D_{\alpha}$ is given by the integral
\[\int_{\DD}|g(z)|^2(1-|z|^2)^{-1-\alpha}dA,\]
and so the spaces $D_{\alpha}$ coincide (as sets) with the standard weighted Bergman spaces discussed in \cite{HKZ00,DS04} and also studied in \cite{W02,W03,CW13}, among other references.

We return to the functions $f_N=(1+z)^N$ and set $N=3$. A direct computation shows that
the first optimal approximant to $1/(1+z)^3$ in $D_{-2}$ vanishes at
\[z_0=\frac{\|z(1+z)^3\|^2_{-2}}{\langle(1+z)^3, z(1+z)^3\rangle_{-2}}=\frac{741}{755}=
0.981\ldots,\]
a point inside the disk. By differentiating the function
\[\alpha\mapsto \frac{\|z(1+z)^3\|^2_{\alpha}}{\langle(1+z)^3, z(1+z)^3\rangle_{\alpha}}\]
with respect to $\alpha$, we see that $z_0$ is increasing on the interval $(-\infty,-2]$, and
hence the optimal approximant $p_1$ to $1/(1+z)^3$ has a zero in $\DD$, for any $D_{\alpha}$ with $\alpha\leq -2$. In fact, by choosing $N=N(\alpha)$ large enough we can produce an extraneous zero also for the range $-2<\alpha<1$. We omit the details.

Straight-forward linear algebra computations produce the first few optimal approximants
to $1/(1+z)^3$:
\[p_1=\frac{741}{1694}\left(1-\frac{775}{741}z\right), \quad
p_2=\frac{961}{1638}\left(1-\frac{1571}{961}z+\frac{1032}{961}z^2\right),\]
and
\[p_3=\frac{571}{826}\left(1-\frac{3427}{1713}z+\frac{1182}{571}z^2-\frac{1862}{1713}z^3\right).\]
It can be checked that the zero sets $\mathcal{Z}(p_n)$, $n=1,2,3$, are all contained in the unit disk; see Figure \ref{extraneousfig}.

The second source of examples is the family of functions
\[f_{\eta}=\frac{1+z}{(1-z)^{\eta}}, \quad \eta>0.\]

We have
\[f_1=\frac{1+z}{1-z}=1+2\sum_{k=1}^{\infty}z^k\]
and we see that $f_1\in D_{-2}$. Moreover, $f_1$ is cyclic as a product of the cyclic multiplier $1+z$ and
the function $1/(1-z)$, which is cyclic in $H^p$ for all $p<1$, and hence also in $D_{-2}$. Using Euler's formula $\sum_{k=1}^{\infty}k^{-2}=\pi^2/6$ to compute $\|zf_1\|^2_{-2}$ and $\langle f_1, zf_1\rangle_{-2}$, we find that
the first-order optimal approximant to $1/f_1$ in $D_{-2}$ vanishes at
\[z_1=\frac{8\pi^2-57}{8\pi^2-54}=0.879...\]
\end{example}

\begin{example}[(Extraneous zeros in the Bergman space)]\label{bigex}
\begin{figure}
\includegraphics[width=0.22 \textwidth]{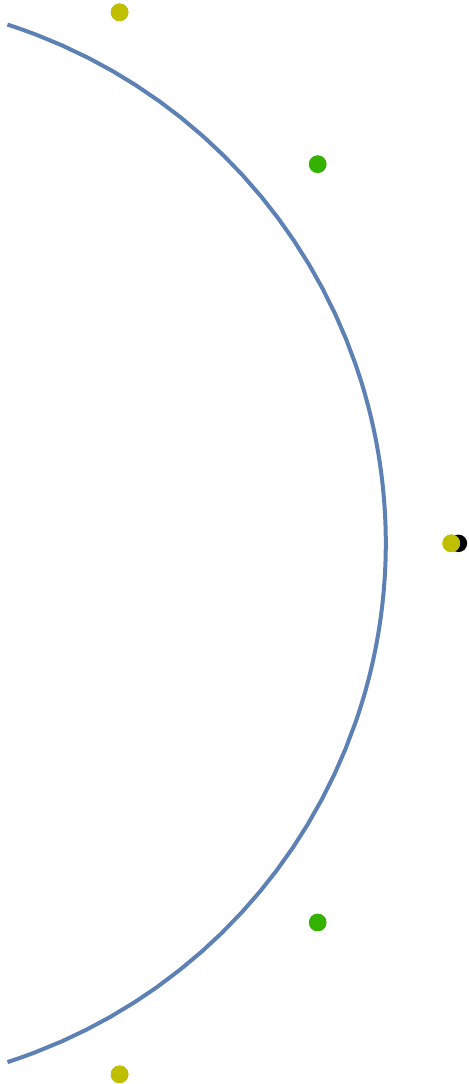}
\caption{Zeros of the optimal approximants $p_1$ (black), $p_2$ (green), and $p_3$ (gold) associated $(1+z)^3$ in the unweighted Bergman space $D_{-1}$.}
\label{nonextraneousfig}
\end{figure}
It can be checked, again by hand, that the zeros of the first few optimal approximants to $1/(1+z)^3$ in the Bergman space $D_{-1}$ are in the complement of the unit disk, see Figure \ref{nonextraneousfig}.  In fact, one can show that $\|zf_N\|^2_{-1}/\langle f_N,zf_N\rangle_{-1}>1$ for all $N\in \NN$.

However, this is not always the case!  Before presenting a specific example, let us give a heuristic explanation for why the zeros of optimal approximants may move inside $\DD$ for Bergman type spaces.  Let $f$ be a cyclic function in the Bergman space $D_{-1}$, say, and define $g(z) = z f(z)$, assuming the normalization $\|g\|_{-1} = 1$.  Then by \eqref{zeroformula}, we need to find $f$ such that $ |z_1| = \frac{1}{|\langle f, zf \rangle_{\alpha}|} < 1$, or equivalently,
\begin{equation*}
\left| \int_{\DD} \frac{1}{z} |g|^2 d A \right| > 1.
\end{equation*}
Letting $h$ be defined by $g= \sqrt{2} \, z \, h$, this equation becomes
\begin{equation}\label{centerofmass}
\left| \int_{\DD} \bar{z} |h|^2 d A \right| > 1/2,
\end{equation}
where, since $\sqrt{2} \, z$ is a contractive divisor and hence an expansive multiplier in all Bergman spaces with logarithmically subharmonic weight (cf. \cite{DKS96,DS04}), we have that
$ \|h\|_{-1} \leq 1.$  In other words, we are looking for $h$ such that the measure $d\mu:= |h|^2 dA$ has total mass at most $1$ but has center of mass close enough to $1$ to ensure that \eqref{centerofmass} holds.  Thus if we are able to choose $h$ so that $\mu$ is concentrated in the circular segment $ S : = \{ z \in \DD: \mathrm{Re}(z) > 1 - \varepsilon \}$ for small $\varepsilon$, and say $h$ is symmetric with
respect to the $x$-axis, then the center of mass of $\mu$ will be real and close to $1$, so inequality \eqref{centerofmass} will be satisfied.  Starting with
$f(z) = \frac{1}{(1-z)^{\beta}}$ for $0 < \beta < 1$ but sufficiently close to $1$, all the requirements will be fulfilled and the zero of the first approximant will move inside $\DD$.

The following example is adjusted from the above idea to make the calculations come out in essentially closed form.
Specifically, let us consider the function $f_{4/5}=(1+z)/(1-z)^{4/5}.$
We note that $f_{4/5}$ is cyclic in $D_{-1}$ since $1+z$ is a cyclic multiplier, and $1/(1-z)^{4/5}$ is cyclic in the Hardy space
$H^{9/8}$, which is contained in the Bergman space (see \cite{DS04}).

By the binomial theorem, we have $f_{\eta}=1+\sum_{k=1}^{\infty}c_k(\eta)z^k$, with
\[c_k(\eta)=(-1)^k\left[
\left(\begin{array}{c}-\eta\\k\end{array}\right)-
\left(\begin{array}{c}-\eta\\k-1\end{array}\right)\right].\]
Using a computer algebra system, such as {\it Mathematica}, one checks that
\[A(\eta)=\sum_{k=1}^{\infty}\frac{\left(\begin{array}{c}-\eta\\k\end{array}\right)^2}{k+2}=(2-2\eta+\eta^2)
\frac{\Gamma(2-2\eta)}{[\Gamma(3-\eta)]^2}-\frac{1}{2},\]
\[B(\eta)=\sum_{k=1}^{\infty}\frac{\left(\begin{array}{c}-\eta\\k-1\end{array}\right)^2}{k+2}=
\frac{1}{3}\, _3F_2(3, \eta, \eta; 1,4;1),\]
and
\begin{equation*}
C(\eta)=-\sum_{k=1}^{\infty}\frac{\left(\begin{array}{c}-\eta\\k\end{array}\right)\left(\begin{array}{c}-\eta\\k-1\end{array}\right)}{k+2}=\frac{1}{1-\eta}\left(
\frac{\Gamma(2-2\eta)}{\Gamma(1-\eta)\Gamma(2\eta)}- \,  _3F_2(2, \eta-1, \eta; 1,3;1)\right).
\end{equation*}
Here, $_3F_2$ denotes the generalized hypergeometric function.
Evaluating at $\eta=4/5$, expressing everything in terms of gamma functions and repeatedly using the functional equation $\Gamma(x+1)=x\Gamma(x)$,
we find that
\[A(4/5)=\frac{26}{25}\frac{\Gamma(2/5)}{[\Gamma(11/5)]^2},
\quad B(4/5)=\frac{2636}{265}\frac{\Gamma(2/5)}{[\Gamma(16/5)]^2},\]
and
\[C(4/5)=\frac{2464}{625}\frac{\Gamma(2/5)}{[\Gamma(16/5)]^2}.\]
Upon combining, we obtain
\[\|zf_{4/5}\|^2_{-1}=\frac{2142}{125}\frac{\Gamma(2/5)}{[\Gamma(16/5)]^2}.\]
A similar analysis applies to $\langle f_{4/5}, zf_{4/5}\rangle_{-1}$, and we find that
\[\langle f_{4/5}, zf_{4/5}\rangle_{-1}=9\frac{\Gamma(7/5)}{[\Gamma(11/5)]^2}.\] After simplifying the resulting
ratio, we obtain
\[z_1=\frac{\|zf_{4/5}\|^2_{-1}}{\langle f_{4/5}, zf_{4/5}\rangle_{-1}}=\frac{119}{121}=0.983\ldots,\]
and so $p_1$ has a zero in the unit disk, as claimed.
\end{example}
It is possible that, with additional work, one could use $f_{\eta}$ to exhibit extraneous zeros also for $D_{\alpha}$ in the range $-1<\alpha<0$, but this seems more technically challenging.
\begin{remark}
The failure of Bergman space analogs of results for Hardy and Dirichlet spaces is a common occurrence.
One example of this phenomenon that seems relevant is the existence of
non-cyclic invertible functions in the Bergman space that was discovered in \cite{BH97}. This is in contrast to $H^2$ and the
Dirichlet space, where invertibility implies cyclicity. In \cite{BH97}, as
in our Example \ref{bigex}, the source of unexpected bad behavior is not, as one might predict, a ``large" set on the boundary where the function vanishes, but rather the presence of regions of rapid growth of the function.

Another example, close in spirit to the previous example, of how Hardy and Bergman spaces are different can be found in \cite{GPS03}. There, it is shown that while eigenfunctions of a certain restriction operator acting on $H^2$ never vanish on the unit circle, eigenfunctions of the corresponding operator on the Bergman space may indeed vanish on $\TT$. We thank Harold S. Shapiro for pointing out this reference to us.

Viewed in a different light, it is perhaps somewhat surprising that there are extraneous zeros inside the disk
in the case of the unweighted Bergman space. An important step in the construction of contractive divisors
for Bergman spaces (see \cite{HKZ00,DS04,W03}) is to rule out extraneous zeros of a certain extremal function. This can be done for the Bergman space, and more generally for the weighted spaces $D_{\alpha}$ for $-2\leq \alpha<0$, but extraneous zeros do appear when $\alpha<-2$, see \cite{HZ92}. In our case, zeros in the disk are present already for $\alpha=-1$.
\end{remark}
\begin{remark}
Since all of our examples are cyclic vectors, the associated reproducing kernels have to converge, as $n\to \infty$, to the reproducing kernels of the respective $D_{\alpha}$. These latter kernels are zero-free, and hence the zeros of $p_n$ have to leave every closed subset of the unit disk eventually. See \cite{DKS96} for details. It does not seem easy
to determine how fast this happens, or whether there is any monotonicity involved: in principle it could happen that some $p_n$ is zero-free, while some subsequent $p_{n'}$ again vanishes inside $\DD$.
It is known, see \cite[Section 3]{BCP98}, that monotonicity does not hold for zeros of
Taylor polynomials associated with outer functions.
\end{remark}

\begin{acknowledgements}\label{ackref}
Part of this work was carried out while the authors were visiting the Institut Mittag-Leffler (Djursholm, Sweden), thanks to NSF support under the grant DMS1500675. The authors would like to thank the Institute and its staff for their hospitality. CB and DK would like to thank E. Rakhmanov for several illuminating conversations about orthogonal polynomials. DS acknowledges suppport by ERC Grant 2011-ADG-20110209 from EU
programme FP2007-2013 and MEC Projects MTM2014-51824-P and
MTM2011-24606. AS thanks Stefan Richter for a number of inspiring conversations about orthogonal polynomials during
a visit to the University of Tennessee, Knoxville in Spring 2014.
\end{acknowledgements}

\affiliationone{
   C. B\'en\'eteau, D. Khavinson, A.A. Sola \\
   Department of Mathematics\\
      University of South Florida\\
   4202 E Fowler Ave, CMC342\\
   Tampa, FL 33620\\
   USA
   \email{benetea@usf.edu,
   dkhavins@usf.edu, sola@usf.edu}}

\affiliationone{
   C. Liaw\\
   CASPER and Department of Mathematics\\
   Baylor University\\
   One Bear Place \#97328\\
   Waco, TX 76798-7328\\
   USA
   \email{Constanze$\underline{\,\,\,}$Liaw@baylor.edu}}

\affiliationone{
  D. Seco\\
Departament de Matem\`atica Aplicada i An\`alisi\\ Facultat de Matem\`atiques\\
Universitat de Barcelona\\
Gran Via 585\\
08007 Barcelona\\ Spain.
  \email{dseco@mat.uab.cat} }

\end{document}